\documentclass[12pt]{amsart}
\usepackage{amssymb}

\theoremstyle{plain}
\newtheorem{theorem}{Theorem}[section]

\theoremstyle{definition}
\newtheorem{definition}[theorem]{Definition}

\theoremstyle{remark}

\newcommand{\abs}[1]{\lvert#1\rvert}
\newcommand{\norm}[1]{\lVert#1\rVert}

\newcommand{\Bignorm}[1]{\Bigl\lVert#1\Bigr\rVert}

\renewcommand{\mid}{\::\:}

\newcommand{\NN}{\mathbb{N}}

\newcommand{\veps}{\varepsilon}
\newcommand{\tow}{\stackrel{w^*}{\longrightarrow}}

\DeclareMathOperator{\codim}{codim}

\begin{document}
\baselineskip 18pt

\title{The invariant subspace problem for rank one perturbations}

\author[A.~Tcaciuc]{Adi Tcaciuc}
\address[A. Tcaciuc]{Mathematics and Statistics Department,
   Grant MacEwan University, Edmonton, Alberta, Canada T5J
   P2P, Canada}
\email{tcaciuca@macewan.ca}

\thanks{${}^1$ Research supported in part by NSERC (Canada)}

\begin{abstract}
We  show  that  for any  bounded  operator $T$  acting  on  an  infinite
dimensional Banach space there exists an operator $F$ of rank at most one such that $T+F$ has an invariant subspace of infinite dimension and codimension.  We also show that whenever the boundary of the spectrum of $T$ or $T^*$ does not consist entirely of eigenvalues, we can find such rank one perturbations that have arbitrarily small norm. When this spectral condition is not satisfied, we can still find suitable finite rank perturbations of arbitrarily small norm, but not necessarily of rank one.
\end{abstract}

\maketitle

\section{Introduction}\label{intro}

The Invariant Subspace Problem is one of the most famous problem in Operator Theory, and is concerned with the search of non-trivial, closed, invariant subspaces for bounded operators acting on a separable Banach space. Considerable success has been achieved over the years both for the existence of such subspaces for many classes of operators, as well as for non-existence of invariant subspaces for particular examples of operators. However, for the most important case of a separable Hilbert space, the problem is still open. For the remaining of the paper, by ''subspace'' we mean a norm closed subspace.

 For compact operators, von Neumann (for Hilbert spaces, unpublished), and Aronszjin and Smith \cite{AS54} (for Banach spaces) showed that the Invariant Subspace problem has a positive answer.  A remarkable result of the 1970s is Lomonosov's \cite{L73} proof that every operator commuting with a compact operator has an invariant subspace, thus substantially increasing the class of operators that have invariant subspaces.  Enflo \cite{E76, E87} constructed the first examples of a Banach space and a bounded operator on it without invariant subspaces, followed by a construction by Read \cite{R84}. Later Read constructed a wealth of such examples: operators on $l_1$ \cite{R85}, strictly singular operators \cite{R91}, quasinilpotent operators\cite{R97}. All such examples are built on non-reflexive Banach spaces.  The most spectacular result of the last decade is Argyros and  Haydon \cite{AH11} construction of an infinite dimensional Banach space on which every operator is the sum of a compact operator and a scalar multiple of the identity. Therefore, in particular, every bounded operator on this space has an invariant subspace. This is the first known example of a infinite dimensional  Banach space having this property. Later, Argyros and Motakis \cite{AM14} constructed the first example of a reflexive infinite dimensional Banach space on which all bounded operators have invariant subspaces. It is still an open problem whether every infinite dimensional reflexive Banach space has this property.  We refer the reader to the monograph by Radjavi and Rosenthal \cite{RR03} for an overview  and to the book by Chalendar and Partington \cite{CP11}for more recent approaches to the Invariant Subspace Problem.

In this paper we solve in full generality a question closely related to the Invariant Subspace Problem:  \emph{given a bounded operator on a Banach space, can we always find a finite rank perturbation of it that has a ''non-trivial'' invariant subspace?}  More precisely, we prove the following theorem.
\begin{theorem} \label{mainresult}
Let $X$ be a infinite dimensional complex Banach space, and $T$ a bounded operator acting on $X$. Then there exists a bounded operator $F$ of rank at most one such that $T+F$ has a invariant subspace of infinite dimension and codimension.
\end{theorem}

Note that any finite dimensional or finite codimensional subspace is invariant under some suitable finite rank perturbation, thus for this question ''non-trivial'' subspace  means a subspace of infinite dimension and codimension. Such a subspace will be henceforth called a \emph{half-space}. A half-space that is invariant for some finite rank perturbation of $T$ is called \emph{almost-invariant} for $T$ (see Section \ref{def} for an equivalent definition).

 Invariant subspaces for perturbations of bounded operators have been studied for a long time, mostly in the Hilbert space setting. For example, Brown and Pearcy~\cite{BP71} proved that for any $T\in\mathcal{B}(\mathcal{H})$, where $\mathcal{H}$ is an infinite-dimensional separable Hilbert space, and for any $\varepsilon > 0$, there exists a compact operator $K$ with norm at most $\varepsilon$ such that $T+K$ has an invariant half-space. As an immediate consequence of Voiculescu's ~\cite{V76} famous non-commutative Weyl-von Neumann Theorem it follows that there exists a compact operator $K$ such that $T+K$ has a reducing half-space, that is, a half-space that is invariant for both $T+K$ and $(T+K)^*$.

An equivalent formulation of this problem was introduced by Androulakis, Popov, Tcaciuc, and Troitsky \cite{APTT09}, and in the same paper the authors proved that certain weighted shifts admit\emph{ rank one} perturbations that have invariant half-spaces. The question was solved in affirmative for reflexive Banach spaces by Popov and Tcaciuc \cite{PT13}, who showed that for every bounded operator on a reflexive Banach space, some \emph{rank one} perturbation of it has an invariant half-space. In the same paper, for the Hilbert space case,  the authors prove the existence of ''good'' perturbations that are also of small norm. This gives a substantial improvement over the result of Brown and Pearcy mentioned above, by showing that for any bounded operator $T\in\mathcal{B}(\mathcal{H})$, and for any $\veps>0$  there exists a \emph{finite rank} operator $F$ with norm at most $\varepsilon$ such that $T+F$ has an invariant half-space. Moreover, when either the boundary of $T$ or $T^*$ does not consist entirely of eigenvalues, $F$ can be taken to be \emph{rank one}. In \cite{TW17} this result was extended to the reflexive case.

 Partial solutions for general Banach spaces were given by Sirotkin and Wallis, first for weakly-compact operators and for quasinilpotent operators in \cite{SW14}, then for strictly singular operators in \cite{SW16}. In the latter paper they also showed that any bounded operator acting on a Banach space admits a \emph{compact} perturbation that has an invariant half-space. Common almost-invariant half-spaces for algebras of operators have been studied in \cite{P10}, \cite{MPR13}, and \cite{SW16}. In \cite{MPR13} the authors show that whenever a norm closed algebra of operators on a Hilbert space admits a common almost-invariant half-space, then it actually admits a common \emph{invariant} half-space. This result was extended in \cite{SW16} to norm closed algebras of operators on Banach spaces.

 In Section \ref{main} of this paper we solve the problem in full generality.  In Section \ref{section small norm} we refine the method to obtain, under certain spectral assumptions, perturbations small in norm.

\section{Definitions and Preliminaries} \label{def}

For a Banach space $X$, we denote by $\mathcal{B}(X)$ the algebra of all (bounded linear) operators on $X$. When  $T\in{\mathcal B}(X)$, we write $\sigma(T)$, $\sigma_p(T)$,$\sigma_{ess}(T)$,  $\rho(T)$ and $\partial\sigma(T)$ for the spectrum
of~$T$,  point spectrum of $T$, the essential point spectrum of $T$,  the resolvent set of~$T$ and the topological boundary of the spectrum, respectively. The closed span of a set  $\{x_n\}_n$ of vectors in $X$ is denoted by $[x_n]$. A sequence $(x_n)_{n=1}^{\infty}$ in $X$ is called a \emph{basic sequence} if any $x\in[x_n]$ can be written uniquely as $x=\sum_{n=1}^{\infty} a_n x_n$, where the convergence is in norm (see \cite[section 1.a]{LT77} for background on Schauder bases and basic sequences). If $W$ is a subset of $X^*$, the dual of $X$, then the \emph{pre-annihilator} of $W$ in $X$, denoted by $W^\top$, is defined as:

$$
W^{\top}:=\{x\in X : f(x)=0\mbox{ for all } f\in W\}
$$

It is straightforward to verify that $W^{\top}$ is a closed subspace of $X$. The following definition was introduced in \cite{APTT09}, towards an equivalent formulation of the question under consideration.

\begin{definition}
If $T\in\mathcal{B}(X)$ and $Y$ is a subspace of $X$, we say that $Y$ is \emph{almost-invariant} for $T$ if there exists a finite dimensional subspace $E$ of $X$ such that $TY\subseteq Y+E$. The smallest dimension of such an $E$ is called the \emph{defect} of $Y$ for $T$
\end{definition}

It was proved in \cite{APTT09} (see Proposition 1.3) that  a half-space $Y$ is almost-invariant with defect $k$ for a bounded operator $T$, if and only if there exists a rank $k$ operator $F$ such that $Y$ is invariant for $T+F$, in other words the equivalency we mentioned before.

 Note that when $X$ is not separable, any bounded operator $T$ will have an invariant half-space. Indeed, it is sufficient to consider $Y$ the closed span of the set $\{T^n(x_k): n,k\in\mathbb{N}\}$, where $(x_n)_{n=1}^{\infty}$ is a linearly independent sequence in $X$. Clearly $Y$ is infinite dimensional $T$-invariant subspace of $X$, and since it is separable (while $X$ is not), it is also a half-space. Therefore, for the rest of the paper we may only consider separable Banach spaces.

The main result of \cite{APTT09} was the following theorem, which was used in that paper to prove the existence of almost-invariant half-spaces for certain classes of weighted shifts, and it will also be important here. Recall that a sequence $(x_n)$ in a Banach space is called \emph{minimal} if, for every~$k\in\mathbb N$, $x_k$ does not belong to the closed linear span of the set $\{x_n\mid n\neq k\}$ (see also \cite[Section~1.f]{LT77}).

\begin{theorem}\cite[Theorem 3.2]{APTT09}, \cite[Remark 1.3]{MPR12}\label{oldpaper}
  Let $X$ be a Banach space and $T \in {\mathcal B}(X)$ satisfying the
  following conditions:
  \begin{enumerate}
  \item\label{rho} The unbounded component of the resolvent set $\rho(T)$ contains
      $\{z\in\mathbb C\mid 0<\abs{z}<\varepsilon\}$ for some $\varepsilon>0$.
  \item There is a vector $e\in X$ whose orbit $\{T^{n}e\}_{n=0}^{\infty}$ is a minimal sequence.
  \end{enumerate}
  Then $T$ has an almost-invariant half-space with defect at most one.
\end{theorem}

As we mentioned in the Introduction, the almost-invariant half-space problem was solved for reflexive Banach spaces by Popov and Tcaciuc in \cite{PT13}. An important step was the following theorem which proves the existence of almost-invariant half-spaces provided a certain spectral condition holds, and this result will also feature in our proof of the general case:

\begin{theorem}\cite[Theorem 2.3]{PT13} \label{oldmain}
 Let $X$ be an infinite dimensional Banach space and let  $T \in
{\mathcal B}(X)$ such that there exists $\mu\in\partial\sigma(T)$ that is not an eigenvalue. Then $T$ admits an almost-invariant half-space with defect one.
\end{theorem}

\section{Every bounded operator has an almost-invariant half-space}\label{main}

An important ingredient in our proof is the following  $w^*$-analogue of the Bessaga-Pelczynski selection principle. An outline of the proof first appeared in a paper by Johnson and Rosenthal (see Theorem III.1 and Remark III.1 in \cite{JR72}). For a shorter proof see the recent paper of Gonz\'{a}lez and Martinez-Abej\'{o}n \cite{GM12}. Recall that a sequence $(x_n)$ in a Banach space is called \emph{semi-normalized} if $0<\inf\|x_n\|\leq\sup\|x_n\|<\infty$.

\begin{theorem}\cite{JR72}\cite{GM12} \label{JR}
  If $(x^*_n)$ is a semi-normalized, $w^*$-null, sequence in a dual Banach space $X^{*}$, then there exists a basic subsequence $(y^*_n)$ of $(x^*_n)$, and a bounded sequence $(y_n)$ in $X$ such that $y^*_i(y_j)=\delta_{ij}$ for all $1\leq i,j<\infty$.
\end{theorem}

We begin by proving an essential step for the general case, step that deals with the situation when $T^*$ satisfies a spectral condition similar to the one in the hypothesis of Theorem \ref{oldmain}.

\begin{theorem} \label{newmain}
Let $X$ be a separable Banach space and $T\in\mathcal{B}(X)$ a bounded operator such that $\partial\sigma(T^{*})\setminus\sigma_p(T^{*})\neq\emptyset$. Then $T$ has an almost-invariant half-space with defect at most one.
\end{theorem}

\begin{proof}

  Let $\lambda\in\partial\sigma(T^{*})\setminus\sigma_p(T^{*})$ and without loss of generality assume $\lambda=0$, otherwise work with $T-\lambda I$. Let $(\lambda_n)$ be a sequence in the resolvent $\rho(T^*)$ such that $\lambda_n\to 0$. Then we have that $\norm{(\lambda_nI-T^{*})^{-1}}\to\infty$ and, from Uniform boundness principle, it follows that there exists $e^{*}\in X^*$ such that $\norm{(\lambda_nI-T^{*})^{-1}e^*}\to\infty$. Put $h_n^*:=(\lambda_nI-T^{*})^{-1}e^*$ and $x_n^*:=h_n^*/\norm{h_n^*}$. Easy calculations show that
  \begin{equation}\label{inveq}
   T^*x_n^*=\lambda_n x_n^*-\frac{e^*}{\norm{h_n^*}}
  \end{equation}

  Claim 1: $(x_n^*)$ has a subsequence that is $w^*$-null.

   From Banach-Alaoglu we have that $B_{X^*}$, the unit ball of $X^{*}$, is $w^{*}$-compact, and since $X$ is separable, $B_{X^*}$ is also $w^*$-metrizable. Therefore, by passing to a subsequence, we can assume that $x_n^*\tow y^*$ for some $y^*\in X^*$. Remains to show that $y^*=0$. Since $\lambda_n\to 0$,  $x_n^*\tow y^*$, and $\norm{h_n^*}\to\infty$ we have that

\begin{equation}\label{conv}
  T^*x_n^*\tow T^*y^*  \mbox{ and } \lambda_n x_n^*-\frac{e^*}{\norm{h_n^*}}\tow 0
\end{equation}

From $(\ref{inveq})$ and $(\ref{conv})$ it follows that $ T^*y^*=0$. However $0$ is not an eigenvalue, so we must have that $y^{*}=0$ and the Claim 1 is proved.

Claim 2: $(x_n^*)$ has a subsequence $(x_{n_k}^*)$ such that $[x_{n_k}^*]^{\top}$ is a half-space of $X$.

  By passing to a subsequence we can assume that $x_n^*\tow 0$. From Theorem $\ref{JR}$, by passing to a further subsequence, we can assume that   $(x_n^*)$  is a basic sequence and  there exists $(x_n)\subseteq X$ such that $x_n^*(x_k)=\delta_{nk}$ for any $n,k\in\NN$. It is routine to check that both $(x_n)$ and $(x_n^*)$ are linearly independent, and that $[x_{2n+1}]\subseteq [x_{2n}^*]^{\top}$, therefore $[x_{2n}^*]^{\top}$ is infinite dimensional. We also have that for any $n\in\NN$ $x_{2n}^{*}([x_{2n}^*]^{\top})=0$, therefore $[x_{2n}^*]^{\top}$ is infinite codimensional as well and Claim 2 is proved.

  In view of the previous Claims, by passing to a subsequence we may  assume that $(x_n^*)$ is a basic sequence and that $Z:=[x_n^*]^{\top}=[h_n^*]^{\top}$ is a half-space of $X$.

 Note that for any $z\in Z$, and for any $n\in\NN$, we have that
  $$
   h^{*}_n(Tz)=T^{*}h^{*}_n(z)=(\lambda_n h^{*}_n-e^{*})z=\lambda_n h^{*}_n(z)-e^{*}(z)=-e^{*}(z)
  $$

  If $Z\subseteq \ker{e^{*}}$ then we have that for all $n\in\NN$ and for all $z\in Z$, $ h^{*}_n(Tz)=0$. Hence $TZ\subseteq Z$ and we are done.

  Otherwise, we can find $z_0\in Z$ such that $z_0\notin\ker e^{*}$. Put $f:=Tz_0$ and for any $z\in Z$ define a scalar $\alpha_z$ by $\alpha_z:=\frac{e^{*}(z)}{e^{*}(z_0)}$. Then, for any $n\in\NN$ and $z\in Z$ we have:
  \begin{eqnarray*}
    h^{*}_n(Tz-\alpha_z f) &=&  h^{*}_n(Tz-\alpha_z f) \\
    &=& h^{*}_n\left(Tz-\frac{e^{*}(z)}{e^{*}(z_0)}Tz_0\right) \\
    &=& h^{*}_n(Tz)-\frac{e^{*}(z)}{e^{*}(z_0)}h^{*}_n(Tz_0)\\
    &=& -e^{*}(z)+\frac{e^{*}(z)}{e^{*}(z_0)}e^{*}(z_0)=0
  \end{eqnarray*}
Therefore, for any $z\in Z$ we have that $Tz-\alpha_z f\in Z$, so
$$
Tz=Tz-\alpha_z f +\alpha_z f \in Z+[f], \mbox{ for all } z\in Z
$$

It follows that $TZ\subseteq Z+[f]$, hence $Z$ is an almost-invariant half-space for $T$ with defect $[f]$.
\end{proof}

We are now ready to prove the general case, Theorem \ref{mainresult}, which we restate below in the equivalent formulation of almost-invariant half-spaces.

\begin{theorem} \label{general}
Let $X$ be a separable Banach space. Then any bounded operator $T\in\mathcal{B}(X)$ has an almost-invariant half-space with defect at most one.
\end{theorem}

\begin{proof}
If $\partial\sigma(T)\setminus\sigma_p(T)\neq\emptyset$ or $\partial\sigma(T^{*})\setminus\sigma_p(T^{*})\neq\emptyset$ then by applying Theorem \ref{oldmain} or Theorem \ref{newmain}, respectively, we obtain that $T$ has an almost-invariant half-space with defect at most one. Therefore, remains to consider the situation when any value in $\partial{\sigma(T)}=\partial{\sigma(T^*)}$ is an eigenvalue for both $T$ and $T^*$.

Easy calculations show that an eigenvector for $T^*$ cancels any eigenvector of $T$ corresponding to  different eigenvalues. It follows that when $\partial{\sigma(T)}=\partial{\sigma(T^*)}$ is infinite, we can actually build an invariant half-space as span of countably many eigenvectors corresponding to a countably infinite subset of $\partial{\sigma(T)}$ such that the complement in $\partial{\sigma(T)}$ is also infinite (see proof of Theorem 2.7 in \cite{PT13} for details).

Remains to consider the case when $\partial{\sigma(T)}$ is finite. In this situation we have $\partial{\sigma(T)}=\sigma(T)$. We can assume without loss of generality that $\sigma(T)$ is a singleton. Indeed if $\sigma(T)=\{\lambda_1,\lambda_2,\dots,\lambda_n\}$, for each $1\leq i\leq n$ consider the Riesz projection $P_i$ associated to $\lambda_i$. That is,  $P_i^2=P_i$,  $P_iT=TP_i$ (so each $X_i:=P_iX$ is a $T$-invariant subspace of $X$), $\sigma(T|_{P_iX})=\{\lambda_i\}$, and $P_1+P_2+\dots+P_n=I$. It follows that one of the subspaces $X_i$ is infinite dimensional and, for that particular $i$, consider the operator $S:=T|_{X_i}:X_i\to X_i$. If $S$ has an almost-invariant half-space $Y\subseteq X_i$, then the same $Y$ is also an almost-invariant half-space for $T$, with the same defect. Therefore, we may assume $\sigma(T)=\{\lambda\}$ and, by replacing $T$  with $T-\lambda I$, we may also assume $\lambda=0$.

 Next we show that either we can find a vector $z$ such that the orbit $\{T^nz\}$ is a minimal sequence, or there exists an infinite dimensional $T$-invariant subspace $Y$ such that restriction of $T$ to $Y$ has dense range. The argument is similar to the second half of the  proof of Theorem 2.7 in \cite{PT13}, we include it here for the sake of completeness. For any $n\in\mathbb{N}$, denote by $Y_{n}=\overline{T^{n}X}$, with $Y_0:=X$. We have that each  $Y_n$ is invariant under $T$, $Y_{n+1}=\overline{TY_n}$ and $X\supseteq Y_1\supseteq Y_2\supseteq\dots.$ Also note that for any $j,n\in\mathbb{N}$ and any $y\in Y_j$ we have that $T^{n}(y)\in Y_{j+n}$. Note that we can assume each $Y_j$ is infinite dimensional; indeed, otherwise, if $j$ is the smallest index for which $Y_j$ is finite dimensional, then any half-space of $Y_{j-1}$ containing $Y_j$ is an invariant half-space for $T$. If $Y_1$ is of infinite codimension in $X$, then $Y_1$ is an invariant half-space for $X$ and we are done.  Therefore we can assume that $Y_1$ is of finite codimension in $X$, hence complemented in $X$, and we can write $X=Y_1\oplus Z$, where $Z$ is finite dimensional. If $Z=\{0\}$ then $T$ has dense range. Otherwise, let $\{z_1,z_2,\dots z_k\}$ be a basis for $Z$ and assume the orbit $\{T^{n}z_j\}_n$ is not minimal for any $1\leq j\leq k$. For any $1\leq j\leq k$ denote by $p_j$ the smallest index such that $T^{p_j}z_j\in[T^n z_j]_{n\neq p_j}$. It is easy too see that for this choice of $p_j$ we actually have that $T^{p_j}z_j\in[T^n z_j]_{n > p_j}$ (see, e.g. Lemma 2.6 in \cite{PT13}) , thus $T^{p_j}z_j\in Y_{p_j+1}$, for any $1\leq j\leq k$. If we let $p_0:=\max\{p_1,  p_2, \dots p_k\}$, it follows that $T^{p_0}z_j=T^{p_{0}-p_{j}}(T^{p_j}z_j)\in Y_{p_0+1}$ for any $1\leq j\leq k$. Therefore, since $\{z_1,z_2,\dots z_k\}$ is a basis for $Z$, we have that $T^{p_0}z\in Y_{p_0+1}$ for any $z\in Z$. We also have that $T^{p_0}y\in Y_{p_0+1}$ for any $y\in Y_1$, and since $X=Y_1\oplus Z$ it follows that $T_{p_0}x\in Y_{p_0+1}$ for any $x\in X$. This means that $\overline{T^{p_0}X}\subseteq Y_{p_0+1}$, so $Y_{p_0}\subseteq Y_{p_0+1}$. On the other hand, $Y_{p_0+1}\subseteq Y_{p_0}$, therefore $Y_{p_0+1}=Y_{p_0}$ and the last equality means that $T_{|Y_{p_0}}$ has dense range.

If we find a vector $z$ such that the orbit $\{T^nz\}$ is a minimal sequence,  we can apply Theorem \ref{oldpaper} and obtain that $T$ has an almost-invariant half-space with defect at most one. Otherwise, there exists $Y$ an infinite dimensional subspace of $X$ such that $T_{|Y}$ has dense range. Consider $S:=T|_{Y}:Y\to Y$. Since $S$ has dense range it follows that $S^{*}$ is injective. Note that $\sigma(S)=\sigma(T)=\{0\}$, therefore $0\in\sigma(S^*)=\sigma(S)$ is not an eigenvalue. We can now apply Theorem \ref{newmain} to conclude that $S$, hence also $T$,  has an almost-invariant half-space with defect at most one.
\end{proof}

\section{Perturbations of small norm} \label{section small norm}

 We proved in the previous section that for any $T\in\mathcal{B}(X)$ we can find a rank one perturbation $F$ such that $T+F$ has an invariant half-space. In this section we show that, under the same spectral assumptions as in Theorem \ref{newmain},  such $F$ may be chosen to be small in norm. When the spectral conditions are not satisfied we still can find a finite rank perturbation $F$ of small norm, but not necessarily rank one, such that $T+F$ has an invariant half-space. In \cite{TW17} the authors proved the following theorem, which we will also use here.

 \begin{theorem}\cite[Proposition 2.2]{TW17} \label{oldsmallnorm}
 Let $X$ be an infinite dimensional Banach space and let  $T \in
{\mathcal B}(X)$ such that there exists $\mu\in\partial\sigma(T)$ that is not an eigenvalue.  Then for any $\veps>0$ there exists a rank one operator $F$ with $\norm{F}<\veps$ such that $T+F$ has an invariant half-space.
\end{theorem}

Thus, this theorem gives the existence of perturbations of small norm when the boundary of the spectrum of $T$ has non-eigenvalues. We begin by proving a companion theorem to the one above, in the situation when $T^*$ satisfies a similar type of spectral condition.

 \begin{theorem} \label{newsmallnorm}
 Let $X$ be a Banach space and $T\in\mathcal{B}(X)$ a bounded operator such that $\partial\sigma(T^{*})\setminus\sigma_p(T^{*})\neq\emptyset$. Then for any $\veps>0$ there exists a rank one operator $F$ with $\norm{F}<\veps$ such that $T+F$ has an invariant half-space.
 \end{theorem}
\begin{proof}
  Fix $\veps>0$. Let $\lambda\in\partial\sigma(T^{*})\setminus\sigma_p(T^{*})$ and, as before,  without loss of generality assume $\lambda=0$. Given $\lambda_n\in\rho(T^*)$, $\lambda_n \to 0$, consider vectors $e^{*}\in X^*$, $\norm{e^*}=1$, $h_n^*:=(\lambda_nI-T^{*})^{-1}e^*$ and $x_n^*:=h_n^*/\norm{h_n^*}$, and $(x_n)$ in $X$ a  as in the proof of Theorem \ref{newmain}. By passing to a subsequence, consider also $(x_n)$ a bounded sequence in $X$ biorthogonal to $(x_n^*)$, as given by Theorem \ref{JR}. Let $M$ be such that $\norm{x_n}\leq M$ for all $n\in\mathbb{N}$ and by passing to a further  subsequence assume that $\sum_{n=1}^{\infty}\norm{h_n^*}^{-1}<\veps/M$ (recall that $\norm{h_n^*}\to\infty$) , and $Z:=[x_n^*]^{\top}$ is a half-space.

  Define $f\in X$ by
  $$
      f:=\sum_{n=1}^{\infty}\frac{1}{\norm{h_n^*}}x_n
  $$

We have
$$
\Bignorm{\sum_{n=1}^{\infty}\frac{1}{\norm{h_n^*}}x_n}\leq\sum_{n=1}^{\infty}\Bignorm{\frac{1}{\norm{h_n^*}}x_n}\leq\sum_{n=1}^{\infty}\frac{M}{\norm{h_n^*}}\leq M\frac{\veps}{M}=\veps
$$

Therefore $f$ is well defined and $\norm{f}\leq\veps$. Note that for all $n$, the bounded functional $h_n^*$ satisfies
$$
h_n^*(f)=\norm{h_n^*}x_n^*(f)=\norm{h_n^*}x_n^*\left(\sum_{i=1}^{\infty}\frac{1}{\norm{h_i^*}}x_i\right)=\norm{h_n^*}\sum_{i=1}^{\infty}\frac{1}{\norm{h_i^*}}x_n^*(x_i)=1
$$

Consider now the rank one operator $F:=e^*\otimes f$, that is, for any $x\in X$,  $F(x)=e^*(x)f$. We have that $\norm{F}=\norm{e^*}\norm{f}<\veps$ and will show that $Z$ is an invariant half-space for $T+F$. To this end, it is enough to show that for any $z\in Z$, and any $n\in\mathbb{N}$, we have that $h_n^*(Tz+Fz)=0$. Indeed:

\begin{eqnarray*}
h_n^*(Tz+Fz)&=&h_n^*(Tz)+h_n^*(Fz)=T^*h_n^*(z)+e^*(z)h_n^*(f)\\
&=&\lambda_n h_n^*(z)-e^*(z)+e^*(z)\\
&=&\lambda_n h_n^*(z)=0.
\end{eqnarray*}

Therefore $(T+F)(Z)\subseteq Z$ and this concludes the proof.
\end{proof}

Next we will prove the result in its full generality, when no assumptions on the spectrum are made.

\begin{theorem}
 Let $X$ be a Banach space and $T\in\mathcal{B}(X)$ a bounded operator. Then for any $\veps>0$ there exists a finite rank operator $F$ with $\norm{F}<\veps$ such that $T+F$ has an invariant half-space. Moreover, if $\partial\sigma(T)\setminus\sigma_p(T)\neq\emptyset$ or $\partial\sigma(T^{*})\setminus\sigma_p(T^{*})\neq\emptyset$, $F$ can be taken to be rank one.
\end{theorem}

\begin{proof}
Fix $\veps>0$. Note that the ''moreover'' part is simply Theorem \ref{oldsmallnorm} when $\partial\sigma(T)\setminus\sigma_p(T)\neq\emptyset$ and Theorem \ref{newsmallnorm} when $\partial\sigma(T^{*})\setminus\sigma_p(T^{*})\neq\emptyset$. Remains to consider the case when $\partial\sigma(T)=\partial\sigma(T^*)$ consist only of eigenvalues. If these sets are infinite, the same argument as in the proof of Theorem 2.7 in \cite{PT13} (also used in Theorem \ref{general} in the previous section ) shows that $T$ actually has an \emph{invariant} half-space. When $\partial\sigma(T)=\partial\sigma(T^*)$ is finite, we can assume as we did in the proof of Theorem \ref{general} that $T$ is quasinilpotent, and $0$ is an eigenvalue for both $T$ and $T^*$.

Denote by $N$ the kernel of $T$ and by $R$ the closure of the range of $T$. If $N$ is infinite dimensional, then any subspace of $N$ that is a half-space will be an invariant half-space for $T$.  Since $T^*$ is not injective it follows that the range of $T$ is not dense in $X$. Clearly $R$ is infinite dimensional, and if it is infinite codimensional as well, then $R$ is an invariant half-space for $T$. Therefore we may assume that $N$ is finite dimensional and $R$ is finite codimensional. Denote by $n:=\dim(N)$ and by $m:=\codim(R)$, and write $X=N\oplus Y$ and $X=R\oplus Z$. Fix bases $\{f_1, f_2, \dots, f_n\}$ of $N$ and $\{g_1, g_2, \dots, g_ m\}$ of $Z$. We will consider separately the cases $n\leq m$ and $n>m$.

If $n\leq m$,  consider the rank $n$ operator $G:N\to Z$ defined by $G(f_i)=g_i$, for any $1\leq i\leq n$. Extend $G$ to $X$ by letting $G|_{Y}=0$. It is easy to verify that for any scalar $\alpha\neq 0$, $T+\alpha G$ is injective. Recall that the essential spectrum is stable under compact perturbations, and that the spectrum of a compact perturbation of a quasinilpotent operator is at most countable, with $0$ the only possible accumulation point (see e.g. \cite{AA02}, Corollary 7.50)). It follows that $0\in\partial\sigma(T+\alpha G)$ and since $T+\alpha G$ is injective, $0$ is not an eigenvalue for $T+\alpha G$. Choose $\alpha>0$ such that $\norm{\alpha G}<\veps/2$. We can apply Theorem \ref{oldsmallnorm} for  $T+\alpha G$ and find $F_0\in\mathcal{B}(X)$ a rank one operator such that $\norm{F_0}<\veps/2$ and $T+\alpha G + F_0$ has an invariant half-space. Then $F:=\alpha G +F_0$ is an operator of rank $n+1$ that satisfies the conclusion.

If $m<n$, consider the rank $m$ operator $G:N\to Z$ defined by $G(f_i)=g_i$ for any $1\leq i\leq m$, and $G(f_i)=0$ for any $m<i\leq n$. Extend  $G$ to a rank $m$ operator on $X=N\oplus Y$ by letting $G|_{Y}=0$. It follows easily that for any scalar $\alpha\neq 0$, $T+\alpha G$ has dense range. The same argument as in the previous paragraph gives that $0\in\partial\sigma(T+\alpha G)$ for any $\alpha\neq 0$. Since $T+\alpha G$ has dense range, it follows that $(T+\alpha G)^*$ is injective and $0\in\partial\sigma(T+\alpha G)^*$ is not an eigenvalue for $(T+\alpha G)^*$ . Pick $\alpha>0$ such that $\norm{\alpha G}<\veps/2$, and apply Theorem \ref{newsmallnorm} for $T+\alpha G$. As before, we can find $F_0\in\mathcal{B}(X)$ a rank one operator such that $\norm{F_0}<\veps/2$ and $T+\alpha G + F_0$ has an invariant half-space. Setting $F:=\alpha G +F_0$ we obtain the conclusion, and this ends the proof.

\end{proof}

\end{document}